\documentclass[a4paper,10pt]{amsart}
\usepackage[utf8]{inputenc}

\usepackage{graphicx,amssymb,amsmath,amsfonts,amsthm,amscd,hyperref,textcomp,relsize}
\usepackage{tikz-cd}
\usetikzlibrary{babel} 

\newtheorem{proposition}{Proposition}

\newtheorem{lemma}{Lemma}
\newtheorem{corollary}{Corollary}
\newtheorem{theorem}{Theorem}

\theoremstyle{remark}

\newcommand{\CC}{\mathbb C}

\newcommand{\QQ}{\mathbb Q}
\newcommand{\Fp}{{\mathbb F}_p}
\newcommand{\Fq}{{\mathbb F}_q}

\newcommand{\Ql}{\bar{\mathbb Q}_\ell}
\newcommand{\Dbc}{D^b_c}

\newcommand{\FF}{\mathcal F}
\newcommand{\GGG}{\mathcal G}
\newcommand{\HH}{\mathcal H}
\newcommand{\LL}{\mathcal L}
\newcommand{\R}{\mathrm R}

\newcommand{\GG}{\mathbb G}
\newcommand{\Gm}{{\mathbb G}_m}
\newcommand{\Gmk}{{\mathbb G}_{m,k}}
\newcommand{\Gmkk}{{\mathbb G}_{m,\bar k}}
\newcommand{\ZZ}{\mathbb Z}

\newcommand{\PP}{\mathbb P}
\newcommand{\HHH}{\mathrm H}

\newcommand{\KK}{\mathcal K}

\newcommand{\GL}{\mathrm{GL}}

\newcommand{\aaa}{\mathbf a}
\newcommand{\bbb}{\mathbf b}
\newcommand{\ccc}{\mathbf c}
\newcommand{\eee}{\mathbf e}
\newcommand{\ttt}{\mathbf t}
\newcommand{\uuu}{\mathbf u}
\newcommand{\vvv}{\mathbf v}
\newcommand{\Char}{\mathbf{Char}}

\title{Equidistribution and independence of Gauss sums}
\author{Antonio Rojas-Le\'on}
\address{Departamento de Álgebra \\
         Facultad de Matemáticas \\
         c/Tarfia s/n, 41012 Sevilla, SPAIN}
\email{arojas@us.es}

\begin{document}

\begin{abstract}
We prove a general independent equidistribution result for Gauss sums associated to $n$ monomials in $r$ variable multiplicative characters over a finite field, which generalizes several previous equidistribution results for Gauss and Jacobi sums. As an application, we show that any relation satisfied by these Gauss sums must be a combination of the conjugation relation $G(\chi)G(\overline\chi)=\pm q$, Galois conjugation invariance and the Hasse-Davenport product formula.
\end{abstract}

\maketitle

\renewcommand{\thefootnote}{}
\footnote{ORCID code: 0000-0003-1683-9487}
\footnote{Partially supported by PID2020-114613GB-I00 (Ministerio de Ciencia e Innovaci\'on) and P20-01056 (Consejer\'{\i}a de Econom\'{\i}a, Conocimiento, Empresas y Universidad de la Junta de Andaluc\'{\i}a and FEDER)}

\section{Introduction}

Let $k=\Fq$ be a finite field of characteristic $p$. For every multiplicative character $\chi:k^\times\to\CC^\times$ the associated Gauss sum is defined as
$$
G(\chi):=-\sum_{t\in k^\times}\chi(t)\psi(t)
$$
where $\psi:k\to\CC^\times$ is the additive character given by
$$
\psi(t)=\exp(2\pi i\mathrm{Tr}_{\Fq/\Fp}(t)/p).
$$
It is well known that $G(\mathbf 1)=1$ (where $\mathbf 1$ stands for the trivial character) and, for non-trivial $\chi$, the absolute value of $G(\chi)$ is $\sqrt{q}$. Moreover, we have the following elementary properties, for every non-trivial $\chi$ (see eg. \cite[Theorem 1.1.4]{berndt1998gauss})
\begin{equation}\label{elem1}
 G(\chi)G(\bar\chi)=\chi(-1)q
\end{equation}
\begin{equation}\label{elem2}
 G(\chi^p)=G(\chi)
\end{equation}
and, for $d|q-1$ and $\epsilon$ a multiplicative character of order $d$, the Hasse-Davenport product formula (eg. \cite[Theorem 11.3.5]{berndt1998gauss})
\begin{equation}\label{h-d}
G(\chi^d)=\chi(d^d)\prod_{i=0}^{d-1}\frac{G(\chi\epsilon^i)}{G(\epsilon^i)}.
\end{equation}

Any other non-trivial additive character $\psi':k\to\CC^\times$ is given by $t\mapsto \psi(\alpha t)$ for some $\alpha\in k^\times$, and we can define the associated Gauss sums
$$
G(\alpha,\chi):=-\sum_{t\in k^\times}\chi(t)\psi(\alpha t).
$$
We have the identity $G(\alpha,\chi)=\overline\chi(\alpha)G(\chi)$ (eg. \cite[Theorem 1.1.3]{berndt1998gauss}).

Fix an algebraic closure $\bar k$ of $k$, and for any $m\geq 1$ let $k_m$ be the unique degree $m$ extension of $k$ in $\bar k$. For any character $\chi:k_m^\times\to\CC^\times$ we denote by $G_m(\chi)$ the corresponding Gauss sum over $k_m$. Every character $\chi:k^\times\to\CC^\times$ can be pulled back to a character of $k_m$ by composing with the norm map; we will also denote this character by $\chi$ if there is no ambiguity. The Hasse-Davenport relation (eg. \cite[Theorem 11.5.2]{berndt1998gauss}) states that
$$
G_m(\chi)=G(\chi)^m
$$
(the lack of sign in this formula is due to the negative sign used in our definition of Gauss sums).

The distribution of Gauss sums has been widely studied. It is a consequence of Deligne's bound on Kloosterman sums that, as $q$ increases, the $q-2$ normalized Gauss sums $q^{-1/2}G(\chi)$ for the set of non-trivial characters $\chi$ of $k$ become equidistributed on the unit circle for the probability Haar measure \cite[Th\'eor\`eme 1.3.3.1]{katz1980se}. In \cite[Theorem 9.5]{katz1988gauss}, Katz proves independent equidistribution on $S^1$ as $q$ increases of the Gauss sums $q^{-1/2}G(\chi_1\chi),\ldots,q^{-1/2}G(\chi_n\chi)$, where $\chi_1,\ldots,\chi_n$ are fixed characters and $\chi$ runs through the set of characters different from $\overline\chi_1,\ldots,\overline\chi_n$.

Our main result in this article is a general equidistribution result for Gauss sums associated to monomials in multiplicative characters. Fix a positive integer $r$, let $n\geq 1$, fix $n$ non-zero $r$-tuples $\aaa_1,\ldots,\aaa_n\in\ZZ^r$, $n$ multiplicative characters $\eta_1,\ldots,\eta_n:k^\times\to\CC^\times$ and $n$ elements $\ttt_1,\ldots,\ttt_n\in(k^\times)^r$. For every $m\geq 1$, let $T_m$ be the set of $r$-tuples of multiplicative characters $\chi=(\chi_1,\ldots,\chi_r)$ of $k_m$ (equivalently, the set of characters of $(k_m^\times)^r$) and $S_m\subseteq T_m$ the subset consisting of the $\chi$ such that $\eta_i\chi^{\aaa_i}:=\eta_i\chi_1^{a_{i1}}\cdots\chi_r^{a_{ir}}\neq\mathbf 1$ for every $i=1,\ldots,n$. For every $\chi\in S_m$ we get an element $\Phi_m(\chi)\in (S^1)^n$ given by
$$
\Phi_m(\chi)=(q^{-m/2}\chi(\ttt_1)G_m(\eta_1\chi^{\aaa_1}),\ldots,q^{-m/2}\chi(\ttt_n)G_m(\eta_n\chi^{\aaa_n})).
$$

We say that a non-zero $r$-tuple $\bbb\in\ZZ^r$ is \emph{primitive} if its coordinates are relatively prime and its first non-zero coordinate is positive. Any non-zero $r$-tuple $\aaa\in\ZZ^r$ can be writen uniquely as $\mu\bbb$, where $\mu\in\ZZ\backslash\{0\}$ and $\bbb$ is primitive. We write all $\aaa_i=\mu_i\bbb_i$ in that way.

Let $\Char_k$ be the injective limit of the character groups $\widehat {k_m^\times}$ for $m\geq 1$ via the maps $\widehat {k_m^\times}\to\widehat {k_{md}^\times}$ given by composition with the norm maps $k_{md}^\times\to k_m^\times$ (which can be identified with the set of finite order characters of the tame fundamental group of $\Gmkk$). We denote by $V_k$ the $\QQ$-vector space with basis the set $\Char_k$. For every $i=1,\ldots,n$, let $\vvv_i\in V_k$ be the element $\sum_{\xi^{\mu_i}=\eta_i}\xi$, that is, the sum of all characters $\xi\in\Char_k$ whose $\mu_i$-th power is $\eta_i$. There are exactly $\nu_i$ such characters, where $\nu_i$ is the prime to $p$ part of $\mu_i$.

\begin{theorem}\label{main}
 Suppose that, for every $\bbb\in\ZZ^r$, the set of $\vvv_i$ for the $i=1,\ldots,n$ such that $\bbb_i=\bbb$ is linearly independent in $V_k$. Then the sets $\{\Phi_m(\chi)|\chi\in S_m\}$ become equidistributed in $(S^1)^n$ with respect to the Haar measure as $m\to\infty$.
\end{theorem}

Note that the condition holds, in particular, if all $\bbb_i$ are distinct, that is, if no two $\aaa_i$'s are proportional.

We recover Katz' result \cite[Chapter 9]{katz1988gauss} by taking $r=1$ and $\aaa_i=(1)$ for all $i=1,\ldots,n$.

\begin{corollary}
 Let $\eta_1,\ldots,\eta_n:k^\times\to\CC^\times$ be distinct characters and $\ttt_1,\ldots,\ttt_n\in k^\times$, then the elements
 $$
 (q^{-m/2}\chi(\ttt_1)G_m(\eta_1\chi),\ldots,q^{-m/2}\chi(\ttt_n)G_m(\eta_n\chi))
 $$
 for the $\chi:k_m^\times\to\CC^\times$ such that $\eta_i\chi\neq\mathbf 1$ for all $i$ become equidistributed in $(S^1)^n$ as $m\to\infty$.
\end{corollary}

\begin{proof}
 Here $\vvv_i=\eta_i$ for all $i$, which are clearly linearly independent in $V_k$ as they are distinct characters.
\end{proof}

Another interesting case is when $r=1$ and all $\eta_i$ are trivial:

\begin{corollary}\label{cor2}
 Let $0<d_1<\ldots<d_n$ be prime to $p$ integers and $\ttt_1,\ldots,\ttt_n\in k^\times$, then the elements
 $$
 (q^{-m/2}\chi(\ttt_1)G_m(\chi^{d_1}),\ldots,q^{-m/2}\chi(\ttt_n)G_m(\chi^{d_n}))
 $$
 for the $\chi:k_m^\times\to\CC^\times$ such that $\chi^{\mathrm{lcm}(d_1,\cdots, d_n)}\neq\mathbf 1$ become equidistributed in $(S^1)^n$ as $m\to\infty$.
\end{corollary}

\begin{proof}
 Now $\vvv_i=\sum_{\xi^{d_i}=\mathbf 1}\xi$ are linearly independent, since for all $i$, $\vvv_i$ contains a character of order $d_i$ with non-zero coefficient, so it can not be a linear combination of the $\vvv_j$ for $j<i$, which are themselves linear combinations of characters of order $<d_i$.
\end{proof}

Closely related to Gauss sums are Jacobi sums, defined as
$$
J(\chi_1,\ldots,\chi_n):=(-1)^{n-1}\sum_{x_1+\ldots+x_1=n}\chi_1(x_1)\cdots\chi_n(x_n)
$$
for any characters $\chi_1,\ldots,\chi_n:k^\times\to\CC^\times$. If all $\chi_i$ and their product are non-trivial, we have the identity (eg. \cite[Theorem 10.3.1]{berndt1998gauss})
$$
J(\chi_1,\ldots,\chi_n)=\frac{G(\chi_1)\cdots G(\chi_n)}{G(\chi_1\cdots\chi_n)}
$$
and, in particular, $|J(\chi_1,\ldots,\chi_n)|=q^{(n-1)/2}$. We denote by $J_m(\chi_1,\ldots,\chi_n)$ the corresponding Jacobi sums over $k_m$.

From our theorem we can also deduce some equidistribution results for Jacobi sums. The following was proven for two-variable Jacobi sums in \cite{katz-zheng}, and in general in \cite{lu-zheng-zheng}

\begin{corollary}
 For $n\geq 2$, the elements
 $$
 q^{-m(n-1)/2}J_m(\chi_1,\ldots,\chi_n)
 $$
 for the non-trivial $\chi_1,\ldots,\chi_n:k_m^\times\to\CC^\times$ such that $\chi_1\cdots\chi_n\neq\mathbf 1$ become equidistributed in $S^1$ as $m\to\infty$.
\end{corollary}

\begin{proof}
 By the theorem, the elements
 $$
 (q^{-m/2}G_m(\chi_1),\ldots,q^{-m/2}G_m(\chi_n),q^{-m/2}G_m(\chi_1\cdots\chi_n))
 $$
 become equidistributed in $(S^1)^{n+1}$ with respect to the Haar measure (since $\aaa_1=(1,0,\ldots,0),\ldots,\aaa_n=(0,0,\ldots,1),\aaa_{n+1}=(1,1,\ldots,1)$ are pairwise non-proportional). Then their images by the homomorphism $\varphi:(S^1)^{n+1}\to S^1$ given by $(t_1,\ldots,t_n,t_{n+1})\mapsto t_1\cdots t_n/t_{n+1}$, which are precisely the normalized Jacobi sums, become equidistributed with respect of the direct image by $\varphi$ of the Haar measure, which is the Haar measure on $S^1$.
\end{proof}

We can also fix some of the characters, generalizing the two-variable version proven in \cite[Theorem 1.1, 1.2]{ping-xi}

\begin{corollary}
 Fix $n$ positive integers $e_1,\ldots,e_n$ and an $e_i$-tuple of non-trivial characters $\eta_{i,1},\ldots,\eta_{i,e_i}:k^\times\to\CC^\times$ for every $i=1,\ldots,n$ such that the $\prod_{j=1}^{e_i} \eta_{i,j}$ are distinct for $i=1,\ldots,n$. Then the elements
 $$
 (q^{-m(d+e_1-1)/2}J_m(\chi_1,\ldots,\chi_d,\eta_{1,1},\ldots,\eta_{1,e_1}),\ldots,q^{-m(d+e_n-1)/2}J_m(\chi_1,\ldots,\chi_d,\eta_{n,1},\ldots,\eta_{n,e_n}))
 $$
 for the non-trivial $\chi_1,\ldots,\chi_d:k_m^\times\to\CC^\times$ such that $\chi_1\cdots\chi_d\eta_{i,1}\cdots\eta_{i,e_i}\neq\mathbf 1$ for all $i$ become equidistributed in $(S^1)^n$ as $m\to\infty$.
\end{corollary}

\begin{proof}
 By the theorem, the elements
 $$
 (q^{-m/2}G_m(\chi_1),\ldots,q^{-m/2}G_m(\chi_d),q^{-m/2}G_m(\chi_1\cdots\chi_d\eta_{1,1}\cdots\eta_{1,e_1}),\ldots,q^{-m/2}G_m(\chi_1\cdots\chi_d\eta_{n,1}\cdots\eta_{n,e_n}))
 $$
 become equidistributed in $(S^1)^{d+n}$ with respect to the Haar measure. Then their images by the homomorphism $\varphi:(S^1)^{d+n}\to (S^1)^n$ given by
 $$(t_1,\ldots,t_d,t_{d+1},\ldots,t_{d+n})\mapsto (q^{-{me_i}/2}G_m(\eta_{i,1})\cdots G_m(\eta_{i,e_i})t_1\cdots t_d/t_{d+i})_{i=1}^n,$$
 which are precisely the given $n$-tuples of normalized Jacobi sums, become equidistributed with respect to the direct image by $\varphi$ of the Haar measure, which is the Haar measure on $(S^1)^n$ (since $\varphi$ is the composition of a surjective homomorphism followed by a translation).
\end{proof}

\begin{corollary}
 Let $d_1,\ldots,d_n$ be positive prime to $p$ integers. Then the elements
 $$
 q^{-m(n-1)/2}J_m(\chi^{d_1},\ldots,\chi^{d_n})
 $$
 for the $\chi:k_m^\times\to\CC^\times$ such that $\chi^{d_1\cdots d_n}\neq\mathbf 1$ become equidistributed in $S^1$ as $m\to\infty$. 
\end{corollary}

\begin{proof}
 Let $d_{n_1}<\ldots<d_{n_e}$ be the distinct $d_i$'s, each $d_{n_j}$ appearing $m_j$ times. By Corollary \ref{cor2}, the elements
 $$
 (q^{-m/2}G_m(\chi^{d_{n_1}}),\ldots,q^{-m/2}G_m(\chi^{d_{n_e}}),q^{-m/2}G_m(\chi^{d_1\cdots d_n}))
 $$
 for the $\chi:k_m^\times\to\CC^\times$ such that $\chi^{d_1\cdots d_n}\neq\mathbf 1$ become equidistributed in $(S^1)^{e+1}$ as $m\to\infty$. Then their images by the map $\varphi:(S^1)^{e+1}\to S^1$ given by
 $$(t_1,\ldots,t_e,t_{e+1})\mapsto t_1^{m_1}\cdots t_e^{m_e}/t_{e+1},$$
 which are precisely the normalized Jacobi sums, become equidistributed with respect of the direct image by $\varphi$ of the Haar measure, which is the Haar measure on $S^1$.
\end{proof}

From the main theorem we can easily deduce a more general version where we allow using different additive characters in each coordinate:

\begin{corollary}Under the hypothesis of Theorem \ref{main}, suppose given also elements $\alpha_1,\ldots,\alpha_n\in k^\times$, and let
$$
\Phi_m(\chi)=(q^{-m/2}\chi(\ttt_1)G_m(\alpha_1,\eta_1\chi^{\aaa_1}),\ldots,q^{-m/2}\chi(\ttt_n)G_m(\alpha_n,\eta_n\chi^{\aaa_n})).
$$
Then the sets $\{\Phi_m(\chi)|\chi\in S_m\}$ become equidistributed in $(S^1)^n$ with respect to the Haar measure as $m\to\infty$.
\end{corollary}

\begin{proof}
 We have
 $$
 \Phi_m(\chi)=(q^{-m/2}\chi(\ttt_1)\eta_1\chi^{\aaa_1}(\alpha_1^{-1})G_m(\eta_1\chi^{\aaa_1}),\ldots,q^{-m/2}\chi(\ttt_n)\eta_n\chi^{\aaa_n}(\alpha_n^{-1})G_m(\eta_n\chi^{\aaa_n}))=
 $$
 $$
 =(q^{-m/2}\chi(\ttt_1\alpha_1^{-\aaa_1})\eta_1(\alpha_1^{-1})G_m(\eta_1\chi^{\aaa_1}),\ldots,q^{-m/2}\chi(\ttt_n\alpha_n^{-\aaa_n})\eta_n(\alpha_n^{-1})G_m(\eta_n\chi^{\aaa_n}))
 $$
 which are the translates by $(\eta_1(\alpha_1^{-1}),\ldots,\eta_n(\alpha_n^{-1}))$ of the elements
 $$
 (q^{-m/2}\chi(\ttt_1\alpha_1^{-\aaa_1})G_m(\eta_1\chi^{\aaa_1}),\ldots,q^{-m/2}\chi(\ttt_n\alpha_n^{-\aaa_n})G_m(\eta_n\chi^{\aaa_n})),
 $$
 which are themselves equidistributed by the theorem.
\end{proof}

From Katz' result \cite[Theorem 9.5]{katz1988gauss} and the Hasse-Davenport product formula it can be easily deduced that all monomial relations among Gauss sums of the form $G(\eta\chi^n)$ for different $\eta$ and $n\in\mathbb Z$ which hold for ``almost all'' characters $\chi$ are a combination of the identities (\ref{elem1}), (\ref{elem2}) and (\ref{h-d}) above. The second main result of the article uses Theorem \ref{main} to show that the same holds true in the multi-variable case. The following statement will be made precise in section \ref{sec3}:

\begin{theorem}\label{main2}
 Suppose given $\eta_i$, $\aaa_i$ as in Theorem \ref{main} and integers $\epsilon_i$ for $i=1,\ldots,n$, an element $\ttt\in (k^\times)^r$, a non-zero integer $N$, a subset $U_m\subseteq T_m$ for every $m\geq 1$ and a sequence of complex numbers $\{D_m\}_{m\geq 1}$ such that
 $$
 \lim_{m\to\infty}\frac{|U_m|}{q^m}=1
 $$
 and
 $$
 \left(\chi(\ttt)\prod_{i=1}^{n} G_m(\eta_i\chi^{\aaa_i})^{\epsilon_i}\right)^N=D_m
 $$
 for every $m\geq 0$ and every $\chi\in U_m$. Then the expression
 $$
 \prod_{i=1}^{n} G(\eta_i\chi^{\aaa_i})^{\epsilon_i}
 $$
 is a product of expressions of the form
 $$
 G(\eta\chi^\aaa)G(\overline\eta\chi^{-\aaa}),
 $$
 $$
 G(\eta^p\chi^{p\aaa})^{-1}G(\eta\chi^\aaa)
 $$
 and
 $$
 G(\eta^d\chi^{d\aaa})^{-1}\prod_{\xi^d=\mathbf 1}G(\eta\xi\chi^\aaa)
 $$
 for some $\eta$, $\aaa$ and $d|q-1$.
\end{theorem}

The results in this article can be seen as a case of the $\ell$-adic Mellin transform theory developed by Katz in \cite{katz2012conv} for the one-dimensional torus and generalized to higher-dimensional commutative algebraic groups by Forey, Fresán and Kowalski in \cite{forey-aftoffgvcae}. More precisely, let $G$ be any connected commutative algebraic group over $k$, and $\overline{\mathbf P}(G)$ the Tannakian perverse convolution category of $G_{\bar k}$ as defined in \cite[Chapter 3]{forey-aftoffgvcae}, modulo isomorphism. The subcategory $\overline{\mathbf P}_1(G)$ of $\overline{\mathbf P}(G)$ consisting of tannakian rank 1 objects is an abelian group under convolution. For objects $\mathcal L_1,\ldots,\mathcal L_n$ in $\overline{\mathbf P}_1(G)$, let $\mathcal M=\mathcal L_1\oplus\cdots\oplus\mathcal L_n$. The (arithmetic or geometric) tannakian group $H_i$ of each $\mathcal L_i$ is a closed subgroup of $\GL(1)$ (so either $\GL(1)$ itself or the group of $r_i$-th roots of unity for some $r_i\geq 1$), so the tannakian group $H$ of $\mathcal M$ (that is, the group whose category of representations is the subcategory $\langle\mathcal M\rangle$ of $\overline{\mathbf P}(G)$ tensor-generated by $\mathcal M$) can be viewed as a closed subgroup of $\GL(1)^n$ which is mapped onto $H_i$ via the $i$-th projection $\rho_i:\GL(1)^n\to\GL(1)$ for every $i=1,\ldots,n$. Such subgroup is completely determined by the subgroup $X$ of the group of characters of $\GL(1)^n$ (which can be naturally identified with $\ZZ^n$ where $\mathbf a\in\ZZ^n$ corresponds to the character $\rho_{\mathbf a}(\mathbf t):={\mathbf t}^{\mathbf a}=t_1^{a_1}\cdots t_n^{a_n}$) consisting of the characters that act trivially on $H$, more precisely, $H=\bigcap_{\rho\in X}\mathrm{ker}(\rho)$ (see eg. \cite[3.2.10]{springer}). Now the object of the category $\langle\mathcal M\rangle$ corresponding to (the restriction to $H$ of) the character $\rho_{\mathbf a}=\rho_1^{a_1}\otimes\cdots\otimes\rho_n^{a_n}$ is $\LL_1^{\ast a_1}\ast\cdots\ast\LL_n^{\ast a_n}$ (since convolution ``is'' the tensor operation in the category $\overline{\mathbf P}(G)$). So this restriction is trivial if and only if $\LL_1^{\ast a_1}\ast\cdots\ast\LL_n^{\ast a_n}$ is trivial in $\overline{\mathbf P}(G)$, and we get the following result, which was kindly suggested to us by an anonymous referee and gives the proper high-level picture of the results in this article:

\begin{proposition}
 The Tannakian group of $\mathcal M$ is the subgroup $H$ of $\GL(1)^n$ consisting of the $\mathbf x=(x_1,\ldots,x_n)$ such that $x_1^{a_1}\cdots x_n^{a_n}=1$ for all $\mathbf a\in\mathbb Z^n$ with $\LL_1^{\ast a_1}\ast\cdots\ast\LL_n^{\ast a_n}$ trivial.
\end{proposition}

In other words, the tannakian group of $\mathcal M$ reflects the multiplicative relations among the $\mathcal L_i$ (including a particular $\mathcal L_i$ being of finite order). In particular:

\begin{corollary}
 The (arithmetic or geometric) tannakian group of $\mathcal M$ is the entire $\GL(1)^n$ if and only if $\LL_1^{\ast a_1}\ast\cdots\ast\LL_n^{\ast a_n}$ is (arithmetically or geometrically, repectively) non-trivial for every non-zero $\mathbf a\in\mathbb Z^n$.
\end{corollary}

This is precisely what we prove in this article for the geometric tannakian group of the direct sum of the objects $[\mathbf t_i]^\ast \alpha_{i\ast}(\LL_{\psi}\otimes\LL_{\eta_i})$ for $i=1,\ldots,n$ (where $\alpha_i:\Gm\to\Gm^r$ is the map $t\mapsto t^{\mathbf a_i}$) under the hypothesis of Theorem \ref{main}. We will not make use of this general theory here, and instead follow a higher dimensional analogue of Katz' approach in \cite{katz1988gauss}, replacing the (implicit, as they were formally defined in a later work) use of hypergeometric sheaves by the higher dimensional hypergeometric objects and the explicit description of the abelian group $\overline{\mathbf P}_1(\Gm^r)$ given by Gabber and Loeser in \cite{gabber1996faisceaux}.

In section \ref{review} we review the main results of this theory that we will make use of. In section \ref{proof} we give the proof of Theorem \ref{main}, and in section \ref{sec3} we use it to deduce Theorem \ref{main2}. Finally, in the last section we prove a version on the main theorem in which the fields are allowed to have different characteristics.

We will choose a prime $\ell\neq p$ and work with $\ell$-adic cohomology, and will assume a choice of embedding $\iota:\Ql\to\CC$ that we will use to identify elements of $\Ql$ and $\CC$ without making any further mention to it. When speaking about purity of $\ell$-adic objects, we will mean it with respect to the chosen embedding $\iota$.

The author would like to thank the anonymous referees for their valuable comments on earlier versions of the article.

\section{Hypergeometric perverse sheaves on the torus}\label{review}

The main reference for this section is \cite{gabber1996faisceaux}. Let $\Gmk^r$ be the $r$-dimensional split torus over $k$, and $\Gmkk^r$ its extension of scalars to $\bar k$. Denote by $\Dbc(\Gmk^r,\Ql)$ the derived category of $\ell$-adic sheaves on $\Gmk^r$. We have a (!-)convolution operation $\Dbc(\Gmk^r,\Ql)\times\Dbc(\Gmk^r,\Ql)\to\Dbc(\Gmk^r,\Ql)$ given by
$$
\KK\ast\LL(:=\KK\ast_!\LL)=\R\mu_!(\pi_1^\ast\KK\otimes\pi_2^\ast\LL)
$$
where $\pi_1,\pi_2,\mu:\Gmk^r\times\Gmk^r\to\Gmk^r$ are the projections and the multiplication map respectively.

Let $Perv$ denote the subcategory of $\Dbc(\Gmk^r,\Ql)$ consisting of the perverse objects. Any object $\KK\in Perv$ has Euler characteristic $\chi(\KK)\geq 0$ \cite[Corollaire 3.4.4]{gabber1996faisceaux}, and $\KK$ is said to be \emph{negligible} if $\chi(\KK)=0$. The negligible objects form a thick subcategory  $Perv_0$ of $Perv$; denote the quotient category by $\overline{Perv}$. They also form an ideal for the convolution, and if $\KK$ and $\LL$ are perverse, the $i$-th perverse cohomology objects of $\KK\ast\LL$ are negligible for $i\neq 0$ \cite[Proposition 3.6.4]{gabber1996faisceaux}, so the convolution gives a well defined operation $\overline{Perv}\times\overline{Perv}\to\overline{Perv}$. With this operation, $\overline{Perv}$ becomes a Tannakian category, in which the ``dimension'' of an object is its Euler characteristic.

Let $\psi:k\to\CC^\times$ be the additive character defined in the introduction, and $\chi:k^\times\to\CC^\times$ any multiplicative character. Let $\LL_\psi$ and $\LL_\chi$ be the corresponding (restriction of) Artin-Schreier and Kummer sheaves on $\Gmk$ \cite[1.7]{deligne569application} and $\HH(\psi,\chi):=(\LL_\psi\otimes\LL_\chi)[1]$. For every embedding of tori $i:\Gmk\to\Gmk^r$, the object $i_\ast\HH(\psi,\chi)=i_\ast(\LL_\psi\otimes\LL_\chi)[1]\in\Dbc(\Gmk^r,\Ql)$ is perverse with Euler characteristic $1$, and so is $\delta_{\ttt}\ast i_\ast\HH(\psi,\chi)=\mathrm{trans}_{\ttt\ast}i_\ast\HH(\psi,\chi)$ for any $\ttt\in k^\times$, where $\mathrm{trans}_{\ttt}:\Gmk^r\to\Gmk^r$ is the translation map $\mathbf x\mapsto \ttt\mathbf x$ and $\delta_\ttt$ is the punctual object $\Ql$ supported on $\ttt$ \cite[Proposition 8.1.3]{gabber1996faisceaux}.

By multiplicativity, any convolution of objects of this form is perverse with Euler characteristic 1. Such objects are called \emph{hypergeometric} (not all hypergeometric objects on $\Gmkk^r$ are of this form though, as they may arise from characters which are not defined over $k$).

The isomorphism classes of hypergeometric objects of $\overline{Perv}$ on $\Gmkk^r$ form an abelian group under convolution \cite[Corollaire 8.1.6]{gabber1996faisceaux}. Let $\mathcal S$ be the set of one-dimensional subtori of $\Gmkk^r$. We will identify it with the set of primitive $r$-tuples $\bbb\in\ZZ^r$, the $r$-tuple $\bbb$ corresponding to the image of the embedding $i_\bbb:\Gmkk\to\Gmkk^r$ given by $t\mapsto t^\bbb$. Let also $\mathcal C(\Gmkk)$ denote the set of continuous $\ell$-adic characters of the tame fundamental group of $\Gmkk$. The subset of $\mathcal C(\Gmkk)$ consisting of finite order characters can be identified with the set $\Char_k$ via the chosen embedding $\iota:\Ql\to\CC$. Then by \cite[Th\'eor\`eme 8.6.1]{gabber1996faisceaux} there is an isomorphism $\Psi=(\Psi_1,\Psi_2)$ between the group of isomorphism classes of hypergeometric objects on $\Gmkk^r$ and the product $ (\bar k^\times)^r\times\ZZ^{(\mathcal S\times\mathcal C(\Gmkk))}$, such that $
\Psi_1(\delta_{\ttt}\ast i_{\bbb\ast}\HH(\psi,\chi))=\ttt$
and $\Psi_2(\delta_{\ttt}\ast i_{\bbb\ast}\HH(\psi,\chi))=1\cdot(\bbb,\chi)$.

We will need the following lemma:
\begin{lemma}\label{lemma}
 Let $\KK$ be a hypergeometric object in $\Gmkk^r$ such that $\Psi_2(\KK)\neq 0$. Then $\HH^0(\KK)=0$.
\end{lemma}
\begin{proof}
 Perverse objects have finite length, so by the long exact sequence of cohomology sheaves associated to an exact sequence of perverse objects it suffices to show this for the simple components of $\KK$. Since $\chi(\KK)=1$ and the Euler characteristic is additive, all but one of the simple components of $\KK$ are negligible, and the other one is a simple hypergeometric object $\KK_0$ such that $\Psi_2(\KK_0)=\Psi_2(\KK)$.
 
 All simple perverse objects are of the form $j_{!\ast}(\FF[d])$ for some $0\leq d\leq r$, where $j:V\to\Gmkk^r$ is the inclusion of an irreducible smooth subvariety of dimension $d$ and $\FF$ is an irreducible lisse sheaf on $V$ \cite[Th\'eor\`eme 4.3.1]{beilinson1982faisceaux}. If $d>0$, then $\HH^0(j_{!\ast}(\FF[d]))=0$ by \cite[Corollaire 1.4.24]{beilinson1982faisceaux}. All negligible simple objects must have $d>0$, since otherwise they would be punctual objects, which have Euler characteristic $\geq 1$.
 
 It remains to show that $\HH^0(\KK_0)=0$. Otherwise, $\KK_0$ would have $d=0$, that is, it would be a punctual object $\delta_\ttt$ for some $\ttt\in\bar k$, so $\Psi(\KK_0)=(\ttt,0)$ and, in particular, $\Psi_2(\KK_0)=\Psi_2(\KK)=0$.
 \end{proof}

\section{Proof of Theorem \ref{main}}\label{proof}

This section is devoted to the proof of Theorem \ref{main}. Recall that we have fixed a finite field $k=\Fq$, $n$ non-zero $r$-tuples $\aaa_1,\ldots,\aaa_n\in\ZZ^r$, $n$ multiplicative characters $\eta_1,\ldots,\eta_n:k^\times\to\CC^\times$ and $n$ elements $\ttt_1,\ldots,\ttt_n\in (k^\times)^r$. For every $\chi\in S_m$ the element $\Phi_m(\chi)\in (S^1)^n$ is given by
$$
\Phi_m(\chi)=(q^{-m/2}\chi(\ttt_1)G_m(\eta_1\chi^{\aaa_1}),\ldots,q^{-m/2}\chi(\ttt_n)G_m(\eta_n\chi^{\aaa_n})).
$$
Since $q^{-m/2}\chi(\ttt_i)G_m(\bar\eta_i\chi^{-\aaa_i})=\eta_i(-1)\overline{q^{-m/2}\chi((-1)^{|\aaa|}\ttt_i^ {-1})G_m(\eta_i\chi^{\aaa_i})}$ for every $\chi\in S_m$, we may assume without loss of generality that $\mu_i>0$ for every $i=1,\ldots,n$.

In order to prove the equidistribution of $\Phi_m(\chi)$ as $m\to\infty$ we need to show that, for every continuous function $f:(S^1)^n\to\CC$, we have
$$
\lim_{m\to\infty}|S_m|^{-1}\sum_{\chi\in S_m}f(\Phi_m(\chi))=\int_{(S^1)^n}f d\mu,
$$
where $\mu$ is the Haar measure on $(S^1)^n$. Since $(S^1)^n$ is abelian, every such $f$ is a class function, so it suffices to show this for the traces of irreducible representations of $(S^1)^n$, which are dense in the space of class functions by the Peter-Weyl theorem. These irreducible representations are just the characters
$$
\Lambda_{\ccc}:\ttt=(t_1,\ldots,t_n)\mapsto\ttt^{\ccc}=t_1^{c_1}\cdots t_n^{c_n}
$$
for some $n$-tuple $\ccc:=(c_1,\ldots,c_n)\in\ZZ^n$. Let $\Sigma_m(\Lambda_\ccc)=|S_m|^{-1}\sum_{\chi\in S_m}\Lambda_\ccc(\Phi_m(\chi))$. If $\Lambda_\ccc$ is trivial (that is, if $\ccc=\mathbf 0$), then $\Sigma_m(\Lambda_{\ccc})=1=\int_{(S^1)^n}d\mu$, so let us assume that $\ccc\neq\mathbf 0$. Then, since $\int_{(S^1)^n}\Lambda_c d\mu=0$, we need to show that $\lim_{m\to\infty}\Sigma_m(\Lambda_\ccc)=0$. And this is clearly a consequence of the following

\begin{proposition}\label{previa}
Let $a=\sum_i\min_{j:a_{ij}\neq 0}|a_{ij}|$. There exists a constant $A(\ccc)$ such that, for every $m>\log_q(1+a)$,
$$
|\Sigma_m(\Lambda_\ccc)|\leq\frac{A(\ccc)(q^m-1)^rq^{-m/2}+a(q^m-1)^{r-1}}{(q^m-1)^{r-1}(q^m-1-a)}.
$$
\end{proposition}

\begin{proof}
For the sake of notation simplicity, let us assume $m=1$ and denote $\Sigma_1,S_1,T_1$ and $\Phi_1$ by $\Sigma,S,T$ and $\Phi$ respectively. Write $\epsilon_i=c_i/|c_i|$ for every $i$ such that $c_i\neq 0$, $|\ccc|=\sum_i c_i$ and $||\ccc||=\sum_i |c_i|=\sum_i \epsilon_i c_i$. Since $|G(\eta_i\chi^{\aaa_i})|=\sqrt{q}$ for every $\chi\in S$, we have 

$$
|S|\cdot\Sigma(\Lambda_\ccc)=\sum_{\chi\in S}\Lambda_\ccc(\Phi(\chi))=q^{-|\ccc|/2}\sum_{\chi\in S} (\chi(\ttt_1)G(\eta_1\chi^{\aaa_1}))^{c_1}\cdots (\chi(\ttt_n)G(\eta_n\chi^{\aaa_n}))^{c_n}=
$$
$$
=q^{-||\ccc||/2}\sum_{\chi\in S} \prod_{i:c_i>0}(\chi(\ttt_i)G(\eta_i\chi^{\aaa_i}))^{c_i}\prod_{i:c_i<0}\overline{(\chi(\ttt_i)G(\eta_i\chi^{\aaa_i}))}^{-c_i}.
$$
We split this as 
\begin{equation}\label{split}
|S|\cdot\Sigma(\Lambda_\ccc)=\Sigma_1-\Sigma_2,
\end{equation} where
$$
\Sigma_1=q^{-||\ccc||/2}\sum_{\chi\in T} \prod_{i:c_i>0}(\chi(\ttt_i)G(\eta_i\chi^{\aaa_i}))^{c_i}\prod_{i:c_i<0}\overline{(\chi(\ttt_i)G(\eta_i\chi^{\aaa_i}))}^{-c_i}
$$
and
$$
\Sigma_2=q^{-||\ccc||/2}\sum_{\chi\in T\backslash S} \prod_{i:c_i>0}(\chi(\ttt_i)G(\eta_i\chi^{\aaa_i}))^{c_i}\prod_{i:c_i<0}\overline{(\chi(\ttt_i)G(\eta_i\chi^{\aaa_i}))}^{-c_i}.
$$
We will start by evaluating the first sum:
$$
\Sigma_1=(-1)^{||\ccc||}q^{-||\ccc||/2}\sum_{\chi\in T}\prod_{i:c_i>0}\left(\chi(\ttt_i)\sum_{x\in k^\times}\psi(x)\eta_i\chi^{\aaa_i}(x)\right)^{c_i}\prod_{i:c_i<0}\left(\overline\chi(\ttt_i)\sum_{x\in k^\times}\overline\psi(x)\overline\eta_i\overline\chi^{\aaa_i}(x)\right)^{-c_i}=
$$
$$
=(-1)^{||\ccc||}q^{-||\ccc||/2}\sum_{\chi\in T}\prod_{i=1}^n\left(\chi^{\epsilon_i}(\ttt_i)\sum_{x\in k^\times}\psi^{\epsilon_i}(x)\eta_i^{\epsilon_i}\chi^{\epsilon_i\aaa_i}(x)\right)^{|c_i|}=
$$
$$
=(-1)^{||\ccc||}q^{-||\ccc||/2}\sum_{x_{ij}\in k^\times,1\leq i\leq n,1\leq j\leq |c_i|}\psi(\sum_{i,j}\epsilon_i x_{ij})\sum_{\chi\in T}\prod_{i=1}^n\chi^{c_i}(\ttt_i)\eta_i^{\epsilon_i}\chi^{\epsilon_i\aaa_i}(\prod_{j=1}^{|c_i|}x_{ij})=
$$
$$
=(-1)^{||\ccc||}q^{-||\ccc||/2}\sum_{x_{ij}\in k^\times}\psi(\sum_{i,j}\epsilon_i x_{ij})\prod_{i=1}^n\eta_i^{\epsilon_i}(\prod_{j=1}^{|c_i|}x_{ij})\sum_{\chi\in T}\prod_{i=1}^n\chi^{c_i}(\ttt_i)\chi^{\epsilon_i\aaa_i}(\prod_{j=1}^{|c_i|}x_{ij})=
$$
$$
=(-1)^{||\ccc||}q^{-||\ccc||/2}\sum_{x_{ij}\in k^\times}\psi(\sum_{i,j}\epsilon_ix_{ij})\prod_{i=1}^n\eta_i^{\epsilon_i}(\prod_{j=1}^{|c_i|}x_{ij})\sum_{\chi\in T}\prod_{i=1}^n\prod_{l=1}^r\chi_l^{c_i}(t_{il})\chi_l^{\epsilon_ia_{il}}(\prod_{j=1}^{|c_i|}x_{ij})=
$$
$$
=(-1)^{||\ccc||}q^{-||\ccc||/2}\sum_{x_{ij}\in k^\times}\psi(\sum_{i,j}\epsilon_ix_{ij})\prod_{i=1}^n\eta_i^{\epsilon_i}(\prod_{j=1}^{|c_i|}x_{ij})\prod_{l=1}^r\sum_{\chi_l\in\widehat{k^\times}}\prod_{i=1}^n\chi_l^{c_i}(t_{il})\chi_l^{\epsilon_ia_{il}}(\prod_{j=1}^{|c_i|}x_{ij})=
$$
$$
=(-1)^{||\ccc||}q^{-||\ccc||/2}\sum_{x_{ij}\in k^\times}\psi(\sum_{i,j}\epsilon_ix_{ij})\prod_{i=1}^n\eta_i^{\epsilon_i}(\prod_{j=1}^{|c_i|}x_{ij})\prod_{l=1}^r\sum_{\chi_l\in\widehat{k^\times}}\chi_l(\prod_{i=1}^n\prod_{j=1}^{|c_i|}t_{il}^{\epsilon_i}x_{ij}^{\epsilon_ia_{il}}),
$$
where the character $\chi\in T$ is given by $\chi(\mathbf t)=\prod_{l=1}^r\chi_l(t_l)$ for one-dimensional $\chi_1,\ldots,\chi_r\in\widehat{k^\times}$. The inner sum vanishes unless $\prod_{i=1}^n\prod_{j=1}^{|c_i|}t_{il}^{\epsilon_i}x_{ij}^{\epsilon_ia_{il}}=1$, in which case it is equal to $q-1$, so we get
$$
\Sigma_1=(-1)^{||\ccc||}(q-1)^rq^{-||\ccc||/2}\sum_{\mathbf x \in X}\psi(\sum_{i,j}\epsilon_ix_{ij})\prod_{i=1}^n\eta_i^{\epsilon_i}(\prod_{j=1}^{|c_i|}x_{ij})=
$$
$$
=(-1)^{||\ccc||}(q-1)^rq^{-||\ccc||/2}\sum_{\mathbf x \in X}\prod_{i=1}^n\prod_{j=1}^{|c_i|}\psi^{\epsilon_i}(x_{ij})\eta_i^{\epsilon_i}(x_{ij}),
$$
where $X\subseteq (k^\times)^{||\ccc||}$ is the subset consisting of the $(x_{ij})_{1\leq i\leq n,1\leq j\leq |c_i|}$ such that $\prod_{i=1}^n\prod_{j=1}^{|c_i|}t_{il}^{\epsilon_i}x_{ij}^{\epsilon_ia_{il}}=1$ for every $l=1,\ldots,r$. We can rewrite this sum as
$$
(-1)^{||\ccc||}(q-1)^rq^{-||\ccc||/2}\sum_{\stackrel{\lambda_{ij}\in(k^\times)^r}{\prod_{ij}\lambda_{ij}=\mathbf 1}}\sum_{\ttt_i^{\epsilon_i}x_{ij}^{\epsilon_i\aaa_i}=\lambda_{ij}}\prod_{i=1}^n\prod_{j=1}^{|c_i|}\psi^{\epsilon_i}(x_{ij})\eta_i^{\epsilon_i}(x_{ij})=
$$
$$
=(q-1)^rq^{-||\ccc||/2}\sum_{\stackrel{\lambda_{ij}\in(k^\times)^r}{\prod_{ij}\lambda_{ij}=\mathbf 1}}\prod_{i=1}^n\prod_{j=1}^{|c_i|}\sum_{\stackrel{x\in k^\times}{x^{\epsilon_i\aaa_i}=\lambda_{ij}\ttt_i^{-\epsilon_i}}}-\psi^{\epsilon_i}(x)\eta_i^{\epsilon_i}(x).
$$

Let $\alpha_i:\Gm\to\Gm^r$ be the morphism of tori given by $t\mapsto t^{\aaa_i}:=(t^{a_{i1}},\ldots,t^{a_{ir}})$. Then $\alpha_i$ factors as $\beta_i\circ[\mu_i]$, where $\beta_i:t\mapsto t^{\bbb_i}$ is a closed embedding and $[\mu_i]:\Gm\to\Gm$ is the $\mu_i$-th power map. For every $i$, the function
$$
(\lambda_{ij})_{j=1}^r\mapsto -\sum_{\stackrel{x\in k^\times}{x^{\epsilon_i\aaa_i}=\lambda_{ij}\ttt_i^{-\epsilon_i}}}\psi^{\epsilon_i}(x)\eta_i^{\epsilon_i}(x)=
-\sum_{\stackrel{x\in k^\times}{x^{\aaa_i}=\lambda_{ij}^{\epsilon_i}\ttt_i^{-1}}}\psi^{\epsilon_i}(x)\eta_i^{\epsilon_i}(x)
$$
is the trace function, on $\Gm^r$, of the complex
$$
\delta_{\ttt_i}\ast\alpha_{i\ast} \HH(\psi,\eta_i)=\delta_{\ttt_i}\ast\beta_{i\ast}[\mu_i]_\ast \HH(\psi,\eta_i)
$$
if $\epsilon_i=1$, and of
$$
inv_\ast(\delta_{\ttt_i}\ast \alpha_{i\ast} \HH(\overline\psi,\overline\eta_i))=\delta_{\ttt_i^{-1}}\ast inv_\ast\alpha_{i\ast} D(\HH(\psi,\eta_i))(-1)=
$$
$$
=\delta_{\ttt_i^{-1}}\ast inv_\ast D(\alpha_{i\ast}\HH(\psi,\eta_i))(-1)
=\delta_{\ttt_i^{-1}}\ast inv_\ast D(\beta_{i\ast}[\mu_i]_\ast \HH(\psi,\eta_i))(-1)
$$
if $\epsilon_i=-1$, since $\alpha_{i\ast}$ commutes with duality (being a finite map) and the Verdier dual of $\HH(\psi,\eta_i)$ is $\HH(\overline\psi,\overline\eta_i)(1)$. The object $inv_\ast D(\alpha_{i\ast}\HH(\psi,\eta_i))$ is the Tannakian inverse of $\alpha_{i\ast}\HH(\psi,\eta_i)$ \cite[Corollaire 3.7.6]{gabber1996faisceaux}, so let us denote it by $\alpha_{i\ast}\HH(\psi,\eta_i)^{\ast(-1)}$. Arithmetically, it is pure of weight $-1$, since $\HH(\psi,\eta_i)$ is pure of weight $1$. By the Lefschetz trace formula, we conclude that $\Sigma_1$ is $(q-1)^rq^{-||\ccc||/2}q^{\sum_{c_i<0}|c_i|}=(q-1)^rq^{-|\ccc|/2}$ times the Frobenius trace at $\mathbf t=\mathbf 1$ of the $!$-convolution: 
$$
\KK:=\ast_{i=1}^n(\delta_{\ttt_i^{\epsilon_i}}\ast\alpha_{i\ast} \HH(\psi,\eta_i)^{\ast\epsilon_i})^{\ast |c_i|}=\ast_{i=1}^n(\delta_{\ttt_i^{\epsilon_i}}\ast\beta_{i\ast}[\mu_i]_\ast \HH(\psi,\eta_i)^{\ast\epsilon_i})^{\ast |c_i|}\in\Dbc(\Gm^r,\Ql)
$$
which is a hypergeometric object as seen in the previous section. The class of its pull-back to $\Gmkk^r$ in the group $(\bar k^\times)^r\times\ZZ^{(\mathcal S\times\mathcal C(\Gmkk))}$ of hypergeometric objects in $\Gmkk^r$ (which we denote additively) is
$$
\Psi(\KK)=\sum_{i=1}^n\sum_{j=1}^{|c_i|}\Psi(\delta_{\ttt_i^{\epsilon_i}}\ast\beta_{i\ast}[\mu_i]_\ast \HH(\psi,\eta_i)^{\ast\epsilon_i})=
$$
$$
=\sum_{i=1}^n\epsilon_i|c_i|\Psi(\delta_{\ttt_i}\ast\beta_{i\ast}[\mu_i]_\ast \HH(\psi,\eta_i))=
$$
$$
=\sum_{i=1}^nc_i\left[(\ttt_i,0)+\Psi(\beta_{i\ast}[\mu_i]_\ast \HH(\psi,\eta_i))\right].
$$
Write $\mu_i=p^{\pi_i}\nu_i$, where $\nu_i$ is prime to $p$. Then
$$
[\mu_i]_\ast\HH(\psi,\eta_i)=[\nu_i]_\ast[p^{\pi_i}]_\ast\HH(\psi,\eta_i)\cong[\nu_i]_\ast\HH(\psi,\eta_i^{p^{-\pi_i}})\cong
$$
$$
\cong\delta_{\nu_i^{-\nu_i}}\ast\left(\ast_{\xi^{\nu_i}=\eta_i^{p^{-\pi_i}}}\HH(\psi,\xi)\right)
$$
by \cite[4.3]{katz1988gauss} (for the power of $p$ part) and \cite[8.9.1]{katz1990esa} (for the prime to $p$ part), so
$$
\beta_{i\ast}[\mu_i]_\ast \HH(\psi,\eta_i)\cong\delta_{\nu_i^{-\nu_i\bbb_i}}\ast\left(\ast_{\xi^{\nu_i}=\eta_i^{p^{-\pi_i}}}\beta_{i\ast}\HH(\psi,\xi)\right)
$$
by \cite[8.1.10(2a)]{katz1990esa}, and its class in the group of hypergeometric objects is
$$
\Psi(\beta_{i\ast}[\mu_i]_\ast \HH(\psi,\eta_i))=\left(\nu_i^{-\nu_i\bbb_i},\sum_{\xi^{\nu_i}=\eta_i^{p^{-\pi_i}}}(\bbb_i,\xi)\right)
$$
where, as explained in the previous section, we identify the set $\mathcal S$ of one-dimensional subtori of $\Gmkk^r$ with the set of primitive $\bbb\in\ZZ^r$.

Therefore, we have
$$
\Psi_2(\KK)=\sum_{i=1}^nc_i\sum_{\xi^{\nu_i}=\eta_i^{p^{-\pi_i}}}(\bbb_i,\xi)=\sum_{i=1}^nc_i\sum_{\xi^{\mu_i}=\eta_i}(\bbb_i,\xi)
$$
since $\xi^{\nu_i}=\eta_i^{p^{-\pi_i}}\Leftrightarrow \xi^{\mu_i}=\eta_i$. For every $\bbb\in\mathcal S$ and $\xi\in \Char_k$, the coefficient of $(\bbb,\xi)$ in this sum is
$$
\sum_{\stackrel{i=1}{\xi^{\mu_i}=\eta_i,\bbb=\bbb_i}}^n c_i.
$$
Suppose that $\Psi_2(\KK)=0$. Then
$$
\sum_{\stackrel{i=1}{\xi^{\mu_i}=\eta_i,\bbb=\bbb_i}}^n c_i=0
$$
for every non-zero $\bbb$ and every $\xi\in\Char_k$, so
$$
0=\sum_{\xi\in\Char_k}\left(\sum_{\stackrel{i=1}{\xi^{\mu_i}=\eta_i,\bbb=\bbb_i}}^n c_i\right)\xi=\sum_{\stackrel{i=1}{\bbb=\bbb_i}}^n c_i\sum_{\xi^{\mu_i}=\eta_i}\xi
$$
in the vector space $V_k$ for every non-zero $\bbb\in\ZZ^r$, which contradicts the hypothesis that the elements $\sum_{\xi^{\mu_i}=\eta_i}\xi$ for $i$ such that $\bbb_i=\bbb$ are linearly independent (since at least one $c_i$ is non-zero). Therefore $\Psi_2(\KK)\neq 0$, and then lemma \ref{lemma} implies that $\HH^0(\KK)=0$. We conclude that
$$
|\Sigma_1|=(q-1)^r q^{-|\ccc|/2} |\mathrm{Tr}(Frob_{\mathbf 1}|\KK_{\bar {\mathbf 1}})|\leq
$$
$$
\leq (q-1)^rq^{-|\ccc|/2}\sum_{i=-r}^{-1}|\mathrm{Tr}(Frob_1|\HH^i(\KK)_{\bar {\mathbf 1}})|\leq (q-1)^r A_{\KK} q^{-1/2}
$$
where $A_{\KK}:=\sum_{i}\dim\HH^i(\KK)_{\bar{\mathbf 1}}$, since $\KK$ is mixed of weights $\leq |\ccc|$, being the convolution of $||\ccc||$ pure objects, $\sum_{i:c_i>0}c_i$ of them of weight $1$ and $\sum_{i:c_i<0}-c_i$ of them of weight $-1$, and then $\HH^i(\KK)$ is mixed of weights $\leq |\ccc|+i\leq |\ccc|-1$ for every $i\leq -1$.

We now proceed to estimate the second summand of (\ref{split}). We have
$$
T\backslash S=\{\chi\in T| \eta_i\chi^{\aaa_i}=\mathbf 1 \text{ for some }i=1,\ldots,n\}=\bigcup_{i=1}^n\{\chi\in T| \eta_i\chi^{\aaa_i}=\mathbf 1 \}
$$
and $|G(\eta_i\chi^{\aaa_i})|\leq\sqrt{q}$ for every $\chi\in T$ and $i=1,\ldots,n$, so
$$
|\Sigma_2|\leq |T\backslash S|\leq\sum_{i=1}^n|\{\chi\in T|\eta_i\chi^{\aaa_i}=\mathbf 1\}|
$$
and $|\{\chi|\eta_i\chi^{\aaa_i}=\mathbf 1\}|\leq a_i(q-1)^{r-1}$, where $a_i=\min_j\{|a_{ij}|\text{ for }j\text{ such that }a_{ij}\neq 0\}$ (if $a_i=a_{ij_0}$, for every choice of $\chi_{j}$ for $j\neq j_0$ there are at most $a_i$ choices for $\chi_{j_0}$ such that $\eta_i\chi^{\aaa_i}=\mathbf 1$). Therefore
$$
|\Sigma_2|\leq (\sum_{i=1}^n a_i)(q-1)^{r-1}=a(q-1)^{r-1}.
$$
In particular, we have $|S|=|T|-|T\backslash S|\geq (q-1)^r-a(q-1)^{r-1}$. By (\ref{split}) we conclude that, for $q>1+a$ (which, in particular, implies $S\neq\emptyset$):
$$
|\Sigma(\Lambda_\ccc)|=\frac{|\Sigma_1-\Sigma_2|}{|S|}\leq
\frac{|\Sigma_1|+|\Sigma_2|}{(q-1)^{r-1}(q-1-a)}\leq
$$
$$
\leq\frac{A_\KK (q-1)^r q^{-1/2}+a(q-1)^{r-1}}{(q-1)^{r-1}(q-1-a)},
$$
which concludes the proof of the estimate and therefore of Theorem \ref{main}.
\end{proof}

\section{Independence of Gauss sums}\label{sec3}
In this section, we will apply the equidistribution theorem \ref{main} to show that all (monomial) relations between Gauss sums that hold for ``almost all'' multiplicative characters are a combination of the Hasse-Davenport relation, the conjugation relation $G(\chi)G(\bar\chi)=\chi(-1)q$ and the Galois invariance relation $G(\chi^p)=G(\chi)$.

Let $k=\Fq$ be a finite field of characteristic $p$ as in the previous sections. Let $r$ be a positive integer, and $\GGG$ the free abelian (multiplicative) group with basis the set $\{\eee_{\eta,\aaa}\}$ indexed by the pairs $(\eta,\aaa)$ where $\eta:k^\times\to\CC^\times$ is a multiplicative character and $\aaa\in\ZZ^r$ a non-zero $r$-tuple. Every $r$-tuple $\chi=(\chi_1,\ldots,\chi_r)$ of multiplicative characters of $k$ induces a group homomorphism $ev_\chi:\GGG\to\CC^\times$ that maps $\eee_{\eta,\aaa}$ to the Gauss sum $G(\eta\chi^\aaa)$. More generally, for every $m\geq 1$ and every $r$-tuple $\chi$ of multiplicative characters of $k_m$, we get a homomorphism $ev_{m,\chi}:\GGG\to\CC^\times$ that maps $\eee_{\eta,\aaa}$ to $G_m(\eta\chi^\aaa)$. We define the following elements of $\GGG$:
\begin{enumerate}
 \item Given a character $\eta:k^\times\to\CC^\times$ and a non-zero $\aaa\in\ZZ^r$, let $P(\eta,\aaa):=\eee_{\eta,\aaa}\eee_{\bar\eta,-\aaa}$.
 \item Given a character $\eta:k^\times\to\CC^\times$ and a non-zero $\aaa\in\ZZ^r$, let $Q(\eta,\aaa):=\eee_{\eta^p,p\aaa}^{-1}\eee_{\eta,\aaa}$.
 \item Given a character $\eta:k^\times\to\CC^\times$, a non-zero $\aaa\in\ZZ^r$ and a positive $d|q-1$, let $R(\eta,\aaa,d):=\eee_{\eta^d,d\aaa}^{-1}\prod_{\xi^d=\mathbf 1}\eee_{\eta\xi,\aaa}=\eee_{\eta^d,d\aaa}^{-1}\prod_{\xi^d=\eta^d}\eee_{\xi,\aaa}$.
\end{enumerate}

For every $r$-tuple $\chi$ of multiplicative characters of $k_m$ such that $\eta\chi^\aaa\neq\mathbf 1$, we have
$$
ev_{m,\chi}(P(\eta,\aaa))=G_m(\eta\chi^\aaa)G_m(\bar\eta\bar\chi^\aaa)=\eta\chi^\aaa(-1)q^m=\chi((-1)^\aaa)\eta(-1)^mq^m
$$
(where, in the last product, $\eta$ is seen as a character on $k$) and
$$
ev_{m,\chi}(Q(\eta,\aaa))=G_m((\eta\chi^\aaa)^p)^{-1}G_m(\eta\chi^\aaa)=1,
$$
and for every positive $d|q-1$ and every $r$-tuple $\chi$ such that $\eta^d\chi^{d\aaa}\neq\mathbf 1$, we have
$$
ev_{m,\chi}(R(\eta,\aaa,d))=G_m((\eta\chi^\aaa)^d)^{-1}\prod_{\xi^d=1}G_m(\eta\xi\chi^\aaa)=\chi(d^{-d\aaa})\eta(d^{-d})^m\prod_{\xi^d=\mathbf 1}G_m(\xi)
$$
by the Hasse-Davenport product formula (\ref{h-d}).

Let $\HH\subseteq\GGG$ be the subgroup generated by the $P(\eta,\aaa)$, $Q(\eta,\aaa)$ and $R(\eta,\aaa,d)$ for every $\eta:k^\times\to\CC^\times$, non-zero $\aaa\in\ZZ^r$ and $d|q-1$. If $\mathbf x\in\HH$, from the previous paragraph we deduce that there exists some constants $D$ and $n$ and some $\ttt\in (k^\times)^r$ such that for every $m\geq 1$ and all $\chi$ except at most $n$ of them, $\chi(\ttt)ev_{m,\chi}(\mathbf x)=D^m$: if $x=P(\eta,\aaa)$ we can take $\ttt=(-1)^\aaa$ and $D=\eta(-1)q$, if $x=Q(\eta,\aaa)$ we can take $\ttt=(1,\ldots,1)$ and $D=1$, if $x=R(\eta,\aaa,d)$ we can take $\ttt=d^{d\aaa}$ and $D=\eta(d^{-d})\prod_{\xi^d=\mathbf 1}G(\xi)$, and we conclude by multiplicativity. The main result of this section is a converse of this.
\begin{theorem}
 Let $\mathbf x\in \GGG$ and assume that there exist an element $\ttt\in (k^\times)^r$, a non-zero integer $N$, a subset $U_m\subseteq T_m$ for every $m\geq 1$ and a sequence of complex numbers $\{D_m\}_{m\geq 1}$ such that
 $$
 \lim_{m\to\infty}\frac{|U_m|}{q^m}=1
 $$
 and $(\chi(\ttt)ev_{m,\chi}(\mathbf x))^N=D_m$ for every $m\geq 1$ and every $\chi\in U_m$. Then $\mathbf x\in\HH$.  
\end{theorem}

The theorem says that any relation satisfied by Gauss sums associated to monomials for ``almost all'' $r$-tuples of multiplicative characters must be a combination of the identities (\ref{elem1}), (\ref{elem2}) and (\ref{h-d}).

\begin{proof}
 Let $\mathbf x=\prod_{i=1}^n \eee_{\eta_i,\aaa_i}^{\epsilon_i}$ with $(\eta_i,\aaa_i)$ distinct and $\epsilon_i\in\ZZ\backslash\{0\}$, and let $S_m\subseteq T_m$, as in the introduction, be the set of $\chi$ such that $\eta_i\chi^{\aaa_i}:=\eta_i\chi_1^{a_{i1}}\cdots\chi_r^{a_{ir}}\neq\mathbf 1$ for every $i=1,\ldots,n$. Then $D_m$ must have absolute value $q^{m\epsilon N/2}$ for sufficiently large $m$ (large enough so that $S_m\cap U_m$ is non-empty), where $\epsilon=\sum_i\epsilon_i$, since the Gauss sums associated to non-trivial characters of $k_m$ have absolute value $q^{m/2}$. Write $\aaa_i=\mu_i\bbb_i$ for all $i$, where $\mu\in\ZZ\backslash\{0\}$ and $\bbb_i\in\ZZ^r$ is primitive.
 
 We claim that there exists some $\bbb\in\ZZ^r\backslash\{\mathbf 0\}$ such that the elements $\sum_{\xi^{\mu_i}=\eta_i}\xi\in V_k$ for the $i=1,\ldots,n$ such that $\bbb_i=\bbb$ are linearly dependent. Otherwise, let $\uuu\in (k_{m_0}^\times)^r$ be an element with coordinates in some finite extension $k_{m_0}$ of $k$ such that $\uuu^{\epsilon_1}=\ttt$. By theorem \ref{main} the elements $(q^{-m/2}\chi(\uuu)G_m(\eta_1\chi^{\aaa_1}),q^{-m/2}G_m(\eta_2\chi^{\aaa_2}),\ldots,q^{-m/2}G_m(\eta_n\chi^{\aaa_n}))$ for $\chi\in S_m$ would become equidistributed on $(S^1)^n$ as $m\to\infty$ ($m$ being a multiple of $m_0$). Since the homomorphism $(S^1)^n\to S^1$ given by $(t_1,\ldots,t_n)\mapsto \prod_{i=1}^n t_i^{\epsilon_i}$ maps the Haar measure of $(S^1)^n$ to the Haar measure of $S^1$, we conclude that the elements
 $$
 (q^{-m/2}\chi(\uuu)G_m(\eta_1\chi^{\aaa_1}))^{\epsilon_1}\prod_{i=2}^n (q^{-m/2}G_m(\eta_i\chi^{\aaa_i}))^{\epsilon_i}=
 $$
 $$
 =q^{-m\epsilon/2}\chi(\ttt)\prod_{i=1}^n (G_m(\eta_i\chi^{\aaa_i}))^{\epsilon_i}=
 q^{-m\epsilon/2}\chi(\ttt)  ev_{m,\chi}(\mathbf x)
 $$
 become equidistributed on $S^1$ as $m\to\infty$, and then so do their $N$-th powers (since $t\mapsto t^N$ is a surjective homomorphism). This clearly contradicts the hypothesis that (almost all of) these $N$-th powers coincide.

Multiplying by suitable powers of elements of $\HH$ of the form $\eee_{\eta,\aaa}\eee_{\bar\eta,-\aaa}$, we may assume that $\mu_i>0$ for all $i$. We now proceed by induction on $\mu:=\sum_i(\mu_i-1)=\sum_i\mu_i-n$.

If $\mu=0$, then $\mu_i=1$ for all $i$. By the claim, there must be some $\bbb$ such that the elements $\eta_i\in V_k$ for $i$ such that $\bbb_i=\bbb$ are linearly dependent. But that can only happen if two of them coincide, which contradicts the distinctness of the $(\eta_i,\aaa_i)$.

Let $\mu>0$. If some $\mu_i$ is a multiple of $p$, we can multiply $\mathbf x$ by an element of $\HH$ of the form $\eee_{\eta^p,p\aaa}^{-1}\eee_{\eta,\aaa}$, which decreases $\mu$ by $p-1$, and proceed by induction. So we may assume that all $\mu_i$ are prime to $p$. By the claim, there is some $\bbb$ such that the elements $\sum_{\xi^{\mu_i}=\eta_i}\xi\in V_k$ for the $i=1,\ldots,n$ such that $\bbb_i=\bbb$ are linearly dependent. Pick some non-trivial dependency relation, and let $i$ be such that $\bbb_i=\bbb$, $\sum_{\xi^{\mu_i}=\eta_i}\xi$ appears with non-zero coefficient in it, and $\mu_i$ is the largest among the $i$'s with these properties. 

Suppose that, for some $d|\mu_i$, there were two different $d$-th roots of $\eta_i$ defined over $k$. Then their ratio is a non-trivial character of order $e$ for some $e|d$ which is defined over $k$. We deduce that all $e$-th roots of $\eta_i$ are defined over $k$ (one is obtained by raising a $d$-th root to the $d/e$-th power, and then all others by multiplying by powers of the character of order $e$). We can then multiply $\mathbf x$ by the element $\eee_{\eta_i,\mu_i\bbb_i}^{-1}\prod_{\xi^e=\eta_i}\eee_{\xi,(\mu_i/e)\bbb_i}\in\HH$ or its inverse, which decreases $\sum_i(\mu_i-1)$ by $e-1$, and proceed by induction.

So we may assume that, for every $d|\mu_i$, there is at most one $d$-th root of $\eta_i$ defined over $k$. If there is a $d$-th root and a $d'$-th root, then there is a $\mathrm{lcm}(d,d')$-th root by B\'ezout, so there is some maximal $d|\mu_i$ such that $\eta_i$ has a (unique) $d$-th root $\theta$ defined over $k$ and, for every $e|\mu_i$, $\eta_i$ has an $e$-th root defined over $k$ if and only if $e|d$, in which case the $e$-th root in question is $\theta^{d/e}$.

Pick a character $\xi_0\in\Char_k$ such that $\xi_0^{\mu_i/d}=\theta$ (in particular, $\xi_0^{\mu_i}=\eta_i$) and a character $\epsilon\in\Char_k$ of order $\mu_i$. Then $(\xi_0\epsilon)^{\mu_i}=\eta_i$, so $\xi_0\epsilon$ appears with non-zero coefficient in $\sum_{\xi^{\mu_i}=\eta_i}\xi$. By the linear dependence relation, it must appear in $\sum_{\xi^{\mu_j}=\eta_j}\xi$ for some other $j\neq i$ with $\bbb_j=\bbb$. By the distinctness of the $(\eta_i,\aaa_i)$ we can not have $\mu_i=\mu_j$, so by the maximality of $\mu_i$ we must have $\mu_j<\mu_i$. Since $(\xi_0\epsilon)^{\mu_i}=\eta_i$ and $(\xi_0\epsilon)^{\mu_j}=\eta_j$ are defined over $k$, so is $(\xi_0\epsilon)^{\mu_0}$ where $\mu_0=\mathrm{gcd}(\mu_i,\mu_j)<\mu_i$. Then $(\xi_0\epsilon)^{\mu_0}$ is a $(\mu_i/\mu_0)$-th root of $\eta_i$ defined over $k$, so $(\mu_i/\mu_0)|d$ and $(\xi_0\epsilon)^{\mu_0}=\theta^{d/(\mu_i/\mu_0)}=(\xi_0^{\mu_i/d})^{d/(\mu_i/\mu_0)}=\xi_0^{\mu_0}$. Therefore $\epsilon^{\mu_0}$ is trivial, which contradicts the fact that $\epsilon$ has order $\mu_i$.
\end{proof}

\section{Independence of $p$}
In this final section we will prove a version of Theorem \ref{main} where we allow the fields over which the characters $\chi$ are defined to have different characteristics.

Let $\KK\in\Dbc(\PP^r_k,\Ql)$. The \emph{complexity} $c(\KK)\in\mathbb N$ of $\KK$ is defined in \cite[Definition 3.2]{quantitative} as the maximum, for $0\leq s\leq r$, of the sum of the (geometric) Betti numbers of the restriction of $\KK$ to a generic linear subspace of $\PP^r_k$ of dimension $s$. More generally, for a quasi-projective variety $X$ with an embedding $u:X\to\PP^r_k$, the \emph{complexity} of an object $\KK\in\Dbc(X,\Ql)$ is $c_u(\KK):=c(u_!\KK)$ \cite[Definition 6.3]{quantitative}. We will consider $X=\Gmk^r$ embedded in $\PP^r_k$ in the natural way. The following lemma optimizes the upper bound on the complexity that follows from the general formalism in \cite{quantitative}.

\begin{lemma}
 Let $\eta:k^\times\to\CC^\times$ be a character, $\aaa\in\ZZ^r$ a non-zero $r$-tuple and $\ttt\in (k^\times)^r$. Let $\alpha:\Gmk\to\Gmk^r$ be the homomorphism of tori given by $t\mapsto t^\aaa$. Then the complexity of $\KK_{\ttt,\aaa,\eta}:=\delta_\ttt\ast\alpha_\ast\HH(\psi,\eta)\in\Dbc(\Gmk^r,\Ql)$ is bounded by $2\max_i|a_i|$.
\end{lemma}

\begin{proof}
Since taking convolution with $\delta_\ttt$ is just appying a translation, which preserves the set of generic linear subspaces, we may assume $\ttt=\mathbf 1$.

 In that case, $\KK_{\ttt,\aaa,\eta}$ is supported on a one-dimensional subtorus, so its restriction to a generic linear subspace of $\PP_k^r$ of dimension $s$ is empty except for $s=r,r-1$. For $s=r$, since $\alpha$ is a finite map, we have $\HHH^i(\PP_{\bar k}^r,u_!\KK_{\ttt,\aaa,\eta})=\HHH^i_c(\Gmkk^r,\KK_{\ttt,\aaa,\eta})=\HHH^i_c(\Gmkk,\HH(\psi,\chi))$ which is one-dimensional for $i=0$, and vanishes for all other $i$.
 
 Write $\aaa=\mu\bbb$, where $\bbb\in\ZZ^r$ is primitive. We have $\alpha=\beta\circ [\mu]$, where $\beta:t\mapsto t^\bbb$ is an immersion and $[\mu]$ is the $\mu$-th power map. Then $[\mu]_\ast\HH(\psi,\chi)$ has rank $\leq|\mu|$ (the exact rank is the prime to $p$ part of $|\mu|$). A generic hyperplane of $\PP^r_k$ intersects the image of $\alpha$ (or, equivalently, of $\beta$) in $B$ points, where $B$ is the number of solutions in $\bar k^\times$ of
 $$
 \gamma_1 t^{b_1}+\cdots+\gamma_r t^{b_r}=\gamma_0
 $$
 for generic $\gamma_0,\gamma_1,\ldots,\gamma_r$, which is $\leq \max_ib_i-\min\{0,\min_ib_i\}\leq 2\max_i|b_i|$. At each of these points, the restriction of $\KK_{\ttt,\aaa,\eta}$ is a single object in degree $-1$, of dimension $\leq|\mu|$. So the sum of the Betti numbers of this restriction is bounded by $2|\mu|\max_i|b_i|=2\max_i|a_i|$.
\end{proof}

By \cite[Example 7.19]{quantitative}, there is an absolute constant $N=N(r)$ such that the convolution of $d$ objects on $\Gmk^r$ as in the lemma with $\max_i|a_i|\leq A$ has complexity $\leq N^{d-1}A^d$. By \cite[Theorem 6.19(3)]{quantitative} applied to the projections and the multiplication map $\GG^r_{m,\ZZ[\ell^{-1}]}\times\GG^r_{m,\ZZ[\ell^{-1}]}\to\GG^r_{m,\ZZ[\ell^{-1}]}$ this constant $N$ can be taken to be the same for every finite field $k$ (by varying the $\ell$ if necessary). In particular, if $\KK$ is such a convolution, we have by \cite[Theorem 6.8]{quantitative} and \cite[Theorem 6.19]{quantitative} applied now to the closed immersion $\iota_{\mathbf 1}:\PP^0_{\ZZ[\ell^{-1}]}\hookrightarrow\GG^r_{m,\ZZ[\ell^{-1}]}$ mapping the only point of $\PP^0$ to $\mathbf 1$:
$$
A_\KK:=\sum_i \dim \HH^i(\KK)_{\mathbf 1}=\sum_i \HHH^i_c(\{\mathbf 1\},\iota_{\mathbf 1}^\ast\KK) = c(\iota_{\mathbf 1}^\ast\KK)\leq C  N^{d-1} A^d
$$
for some absolute constant $C=C(r)$. Applying this to the proof of Theorem \ref{main} given in section \ref{proof}, we get
\begin{theorem}
Fix some $A>0$. For every $m\geq 1$, let $q_m$ be a prime power such that $\lim_m q_m\to\infty$. Let $k_m$ be the finite field with $q_m$ elements, and pick $n$ multiplicative characters $\eta_{m,1},\ldots,\eta_{m,n}:k_m^\times\to\CC^\times$, $n$ $r$-tuples $\aaa_{m,1},\ldots,\aaa_{m,n}\in\ZZ^r$ such that $\max_j|a_{mij}|\leq A$ for every $i$,
and $n$ points $\ttt_1,\ldots,\ttt_n\in (k_m^\times)^r$. Let $S_m$ be the set of $r$-tuples of multiplicative characters $\chi_1,\ldots,\chi_r:k_m^\times\to\CC^\times$ such that $\eta_{m,i}\chi^{\aaa_{m,i}}\neq \mathbf 1$ for all $i$, and assume that the linear independency hypothesis of Theorem \ref{main} holds for all $m\geq 1$. Then the elements
$$
\Phi_m(\chi)=(q_m^{-1/2}\chi(\ttt_{m,1})G(\eta_{m,1}\chi^{\aaa_{m,1}}),\ldots,q_m^{-1/2}\chi(\ttt_{m,n})G(\eta_{m,n}\chi^{\aaa_{m,n}}))
$$
for $\chi\in S_m$ become equidistributed on $(S^1)^n$ as $m\to\infty$.
\end{theorem}

\begin{proof}
 As in the proof of Theorem \ref{main}, we need to show that for every non-zero $\ccc\in\ZZ^r$ we have
 $$
 \lim_{m\to\infty}|S_m|^{-1}\sum_{\chi\in S_m}\Lambda_\ccc(\Phi_m(\chi))=0.
 $$
 And, as in proposition \ref{previa}, we get for $q_m>1+a_m$, where $a_m:=\sum_i\min_{j:a_{mij}\neq 0}|a_{mij}|\leq nA$:
 $$
 |S_m|^{-1}\left|\sum_{\chi\in S_m}\Lambda_\ccc(\Phi_m(\chi))\right|\leq\frac{A_{\KK}(q_m-1)^rq_m^{-1/2}+a_m(q_m-1)^{r-1}}{(q_m-1)^{r-1}(q_m-1-a_m)}
 $$
 so
 $$
 |S_m|^{-1}\left|\sum_{\chi\in S_m}\Lambda_\ccc(\Phi_m(\chi))\right|\leq
\frac{CN^{||\ccc||-1}A^{||\ccc||}(q_m-1)^rq_m^{-1/2}+nA(q_m-1)^{r-1}}{(q_m-1)^{r-1}(q_m-1-nA)}\stackrel{m\to\infty}{\longrightarrow}0.
 $$
 \end{proof}
 
 Note that the condition $\max_j|a_{mij}|\leq A$ is necessary: if we take $n=r=1$, $\ttt=\mathbf 1$, $\eta_{m,1}$ any non-trivial character of $k_m^\times$ and $\aaa_{m,1}=(q_m-1)$ for every $m$, then $\Phi_m(\chi)=(q_m^{-1/2}G(\eta_{m,1}))$ is independent of $\chi$ for every $m\geq 1$.

\bibliographystyle{amsalpha}
\bibliography{bibliography}

\end{document}